\documentclass[12pt]{article}

\usepackage[margin=2.75cm]{geometry}
\usepackage{fancyhdr}
\usepackage{setspace}
\usepackage{amsmath}
\usepackage{amsfonts}
\usepackage{hyperref}
\usepackage{amssymb}
\usepackage{mdwlist}
\usepackage{amscd}
\usepackage{lipsum}
\usepackage{mathrsfs}
\usepackage{mdframed}
\usepackage{amsthm}
\usepackage{enumitem}
\usepackage{verbatim}
\usepackage{graphicx}
\usepackage{epstopdf}
\usepackage[utf8]{inputenc}
\usepackage[english]{babel}
\usepackage{float}
\usepackage{tikz}
\usepackage{comment}
\usepackage[style=alphabetic]{biblatex}
\bibliography{refs}
\usepackage{csquotes}
\usepackage{cleveref}

\setenumerate[0]{label=(\alph*)}

\theoremstyle{plain}

\crefname{claim}{claim}{claims}
\newtheorem{thm}{Theorem}
\crefname{thm}{theorem}{theorems}
\newtheorem*{thm*}{Theorem}
\crefname{thm*}{theorem}{theorems}
\newtheorem{lem}[thm]{Lemma}
\crefname{lem}{lemma}{lemmas}

\crefname{conj}{conjecture}{conjectures}

\crefname{prop}{proposition}{propositions}

\crefname{exer}{exercise}{exercises}
\newtheorem*{cor}{Corollary}
\crefname{cor}{corollary}{corollaries}

\theoremstyle{definition}
\newtheorem*{defn}{Definition}
\crefname{defn}{definition}{definitions}

\theoremstyle{remark}
\newtheorem{rmk}{Remark}[section]
\crefname{rmk}{remark}{remarks}

\def\R{\mathbb{R}}

\def\Q{\mathbb{Q}}
\def\Z{\mathbb{Z}}

\def\d{\mathrm{d}}

\def\O{\mathcal{O}}

\DeclareFontFamily{U}{min}{}
\DeclareFontShape{U}{min}{m}{n}{<-> udmj30}{}

\def\id{\operatorname{id}}
\def\from{:}

\newcommand{\paren}[1]{\left( #1 \right)}
\renewcommand{\brack}[1]{\left[ #1 \right]}
\renewcommand{\brace}[1]{\left\{ #1 \right\}}
\renewcommand{\angle}[1]{\left\langle #1 \right\rangle}
\newcommand{\absval}[1]{\left| #1 \right|}
\newcommand{\norm}[1]{\left\| #1 \right\|}

\newcommand{\parity}[1]{\operatorname{par} \paren{ #1 }}

\newcommand{\clGrp}[1]{\overline{\angle{#1}}}

\newcommand{\trk}[1]{\operatorname{t-rk} \paren{ #1 }}
\newcommand{\grk}[1]{\operatorname{g-rk} \paren{ #1 }}

\def\intr{\operatorname{int}}

\newcommand{\orb}[2]{\O_{#1} \paren{#2}}

\newcommand{\Fix}[1]{\operatorname{Fix} \paren{#1}}

\def\diam{\operatorname{diam}}
\def\dhr{{\downharpoonright}}

\def\isom{\cong}

\hypersetup{
    colorlinks=true,
    linkcolor=red,
    urlcolor=blue
}

\title{Generation and genericity of the group of absolutely continuous homeomorphisms of the interval}
\author{\textsc{Dakota Thor Ihli}\thanks{Department of Mathematics, University of Illinois at Urbana-Champaign, Urbana, Illinois. Email: \href{mailto:dihli2@illinois.edu}{\texttt{dihli2@illinois.edu}}.}}
\date{\vspace{-4ex}}

\begin{document}

\maketitle
\begin{abstract}
    We examine the Polish group $H_{+}^{AC}$ of order-preserving self-homeomorphisms $f$ of the interval $\brack{0,1}$ for which both $f$ and $f^{-1}$ are absolutely continuous; in particular, we establish two results. First, we prove that $H_{+}^{AC}$ is topologically $2$-generated; in fact, it is generically $2$-generated, i.e., there is a dense $G_{\delta}$ set of pairs $\paren{f,g} \in H_{+}^{AC} \times H_{+}^{AC}$ for which $\angle{f,g}$ is dense. Secondly, we prove that $H_{+}^{AC}$ admits a dense $G_{\delta}$ conjugacy class, and we explicitly characterize the elements thereof.
\end{abstract}

\section{Introduction}

For a non-degenerate compact interval $I \subseteq \R$, we let $H_{+} \paren{I}$ denote the group of increasing self-homeomorphisms of $I$ --- abbreviated to just $H_{+}$ when $I$ is understood from context. $H_{+}$ becomes a Polish group when equipped with the topology of uniform convergence.

$H_{+}$ admits a rich collection of subgroups which have been given a great deal of attention across a variety of areas. For example, Thompson's group $F$ famously embeds as a dense subgroup of $H_{+}$. Our focus will be on subgroups of $H_{+}$ obtained by specifying some sort of regularity condition, such as a level of differentiability. Such subgroups are often dense in $H_{+}$, and so not Polish in their inherited topologies. However, they are often \textit{Polishable} --- that is, they may be given Polish group topologies whose Borel sets agree with the existing Borel sets inherited as a subspace of $H_{+}$.

The obvious examples of such subgroups are the groups $D_{+}^{k}$ for $k \geq 1$, consisting of $C^{k}$ diffeomorphisms of $I$, equipped with the usual $C^{k}$ topology; as well as the closely related group $D_{+}^{\infty}$, consisting of the smooth diffeomorphisms. Here, we are interested in an intermediate smoothness condition, namely absolute continuity. Solecki \cite{Sol99} introduces the group $H_{+}^{AC}$, consisting of those maps $f \in H_{+}$ such that both $f$ and $f^{-1}$ are absolutely continuous, and proves that $H_{+}^{AC}$ is Polishable.

Later, Cohen \cite{Coh20} proves that $D_{+}^{k + AC}$ --- the group of $k$-times differentiable homeomorphisms whose $k$th derivative is absolutely continuous --- is Polishable.\footnote{Unlike absolute continuity, the notions of bounded variation and Lipschitz continuity do not give Polish groups this way. We refer to \cite{Coh20} for more remarks on this.} Thus, we have a hierarchical sequence of Polish groups, with each dense in the one above and with the topologies getting progressively finer:
$$H_{+} \geq H_{+}^{AC} \geq D_{+}^{1} \geq D_{+}^{1+AC} \geq D_{+}^{2} \geq \cdots \geq D_{+}^{k} \geq D_{+}^{k+AC} \geq D_{+}^{k+1} \geq \cdots \geq D_{+}^{\infty}.$$

As for $H_{+}^{AC}$, Herndon \cite{Her18} shows that $H_{+}^{AC}$ is coarsely bounded, and so has a well-defined but trivial large scale geometry (in the sense of Rosendal's framework for large scale geometry of Polish groups \cite{Ros21}). This appears to be the entire extent of existing literature on the properties of $H_{+}^{AC}$ as a topological group.

In \Cref{sec:FinGenSubgps} we address the topological rank of $H_{+}^{AC}$. Recall that a topological group $G$ is \textit{topologically $n$-generated} if it admits a dense $n$-generated subgroup. Furthermore, if the set of $n$-tuples which generate a dense subgroup is comeagre in the product topology on $G^{n}$, we say that $G$ is \textit{generically $n$-generated}. It is easy to show that $H_{+}^{AC}$ is not topologically $1$-generated, so the following result is the best one could hope for.

\begin{thm*}\label{GroupIsGen2Gen}
    $H_{+}^{AC}$ is generically $2$-generated.
\end{thm*}

(We prove this result as \Cref{HACIsGen2Gen}.)

The proof employs the main techniques of the proof that $H_{+}$ is generically $2$-generated, due to Akhmedov \& Cohen \cite{AC19}, tailored to the absolutely continuous case.

In \Cref{sec:ConjClasses}, we address the generic elements of $H_{+}^{AC}$. Following the convention of Truss \cite{Tru92}, we say that an element of a Polish group is \textit{generic} if its conjugacy class is comeagre. Many Polish groups admit generic elements, and many others do not. Furthermore, even among the Polish groups which admit generic elements, it may be that the generic elements do not admit a (known) description --- a list of dynamical or combinatorial properties that characterize whether or not the element is generic.\footnote{For example, absence of explicit characterizations of generic elements is common when studying automorphism groups of countable first-order structures.} We are able to avoid this unfortunate situation: we prove $H_{+}^{AC}$ admits generics by giving an explicit characterization of which elements are generic.

\begin{thm*}\label{GroupHasGenerics}
    $H_{+}^{AC}$ admits generic elements. In fact, the comeagre conjugacy class of $H_{+}^{AC}$ is exactly the set of $f \in H_{+}^{AC}$ satisfying the following properties:
    \begin{enumerate}[label=(\roman*)]
        \item The set $\Fix{f} := \brace{x \in I : f \paren{x} = x}$ is totally disconnected and has no isolated points (i.e.\ it is a Cantor set);
        \item For any connected components $\paren{a,b}$ and $\paren{c,d}$ of $I \setminus \Fix{f}$, with $b < c$, there are points $x,y \in \paren{b,c}$ with $f \paren{x} < x$ and $f \paren{y} > y$;
        \item $\Fix{f}$ has Lebesgue measure zero.
    \end{enumerate}
\end{thm*}

(We prove this result as \Cref{HACHasGenerics}.)

We remark that the set of all $f \in H_{+}$ satisfying properties (i) and (ii) is the unique dense $G_{\delta}$ conjugacy class in $H_{+}$.\footnote{In \cite[\S 4.2]{Hjo00}, Hjorth identifies this result as folklore.} Indeed, we show that having null sets of fixed points is the only extra information needed to boost ``generic in $H_{+}$'' to ``generic in $H_{+}^{AC}$''.

The proof can be outlined as follows. Properties (i) and (ii) are conjugacy-invariant since they are in $H_{+}$, and (iii) is conjugacy-invariant since $\Fix{gfg^{-1}} = g \brack{\Fix{f}}$, and absolutely continuous images of null sets are null. These properties are $G_{\delta}$ in $H_{+}^{AC}$ since they are $G_{\delta}$ in $H_{+}$, and the inclusion map $H_{+}^{AC} \hookrightarrow H_{+}$ is continuous. Checking that these properties are dense in $H_{+}^{AC}$ constitutes the bulk of \Cref{sec:ConjClasses}. Finally, if $f$ and $g$ satisfy these properties, then (i) and (ii) imply $g = \hat{h} f \hat{h}^{-1}$ for some $\hat{h} \in H_{+}$; we use this map $\hat{h}$ as a framework to build a map $h \in H_{+}^{AC}$ such that $h = h f h^{-1}$, using property (iii) to verify absolute continuity of $h$.

\section{Preliminaries}\label{sec:Preliminaries}

Throughout, $I$ will denote a compact interval in $\R$. All measure-theoretic notions (almost-everywhere, etc.)\ are taken with respect to the Lebesgue measure $\lambda$, unless otherwise stated.

\begin{defn}
    We say a function $f \from I \to \R$ is \textbf{absolutely continuous} if for all $\epsilon > 0$, there exists $\delta > 0$ such that for any pairwise disjoint finite sequence of intervals $\paren{\paren{a_{i}, b_{i}}}_{i < n}$ such that $\sum_{i < n} b_{i} - a_{i} \leq \delta$, we have $\sum_{i < n} \absval{f \paren{b_{i}} - f \paren{a_{i}}} \leq \epsilon$.

    We let $AC \paren{I}$ denote the set of absolutely continuous functions $f \from I \to \R$.
\end{defn}

We collect here a brief overview of the needed background on absolutely continuous functions, without proofs. For more details, we refer the reader to a standard text on measure theory, for example \cite[\S 3.5]{Fol13}.

\begin{itemize}
    \item Any Lipschitz continuous function is absolutely continuous, and any absolutely continuous function is uniformly continuous.
    \item A map $f \from I \to \R$ is absolutely continuous if and only if $f$ is differentiable almost-everywhere, $f' \in L_{1} \paren{I}$, and $f \paren{b} - f \paren{a} = \int_{a}^{b} f' \paren{t} \, \d t$ for all $a,b \in I$ (i.e.\ the fundamental theorem of calculus holds).
    \item Any monotone function $f \from I \to \R$ is differentiable almost everywhere, and the derivative $f'$ is $L_{1}$. Thus, if $f \in H_{+}$, then to check $f \in AC \paren{I}$ it suffices to check the fundamental theorem of calculus.
    \item The canonical non-example is given as follows: let $\mu$ be a non-atomic Borel probability measure on $\brack{0,1}$ whose support is Lebesgue null (for example, the ``coin-flip'' measure on the classical Cantor set), and define $s \from \brack{0,1} \to \R$ by $s \paren{x} := \mu \paren{\brack{0,x}}$. Then $s'$ exists and equals $0$ almost everywhere, but $s$ is not absolutely continuous since $\int_{0}^{1} s' \; \d \lambda = 0 \neq 1 = s \paren{1} - s \paren{0}$.
\end{itemize}

\begin{defn}
    Let $H_{+}^{AC} \paren{I}$ denote the subgroup of $H_{+} \paren{I}$ consisting of maps $f$ for which both $f \in AC \paren{I}$ and $f^{-1} \in AC \paren{I}$. (As with $H_{+}$, we simply write $H_{+}^{AC}$ whenever $I$ is clear from context.)
\end{defn}

\begin{rmk}
    The requirement that both $f$ and $f^{-1}$ are absolutely continuous is necessary. Indeed, define $f \paren{x} := \frac{s \paren{x} + x}{2}$, where $s$ is the non-absolutely continuous map defined above. Then $f \in H_{+} \paren{\brack{0,1}}$; moreover, $f \notin AC \paren{\brack{0,1}}$, but $f^{-1}$ is Lipschitz, and so $f^{-1} \in AC \paren{\brack{0,1}}$.
\end{rmk}

\begin{defn}
    Given $f,g \in AC \paren{I}$, and $a,b \in I$ with $a \leq b$, we define: $$\rho_{a}^{b} \paren{f,g} := \int_{a}^{b} \absval{f' - g'} \d \lambda.$$
    If $I = \brack{a,b}$ we may write $\rho_{I}$, or simply $\rho$ if the interval is clear from context.
\end{defn}

By the usual properties of integrals, for any $f,g \in AC \paren{I}$ we have $\rho_{a}^{c} \paren{f,g} = \rho_{a}^{b} \paren{f,g} + \rho_{b}^{c} \paren{f,g}$ whenever $a \leq b \leq c$. It is straightforward to verify that $\rho_{I}$ is a pseudometric on $AC \paren{I}$, where $\rho_{I} \paren{f,g} = 0$ precisely when $f - g$ is constant. (Moreover, $f$ and $g$ will agree on the endpoints in all the cases we will be interested in. In particular, $\rho$ is a metric on $H_{+}^{AC}$.)

\begin{thm}\cite[Lemma 2.4]{Sol99}
    The metric $\rho$ gives $H_{+}^{AC}$ a Polish group topology, whose Borel sets coincide with the ones inherited from $H_{+}$.
\end{thm}

The topology generated by $\rho$ is strictly finer than the topology inherited as a subspace of $H_{+}$. Indeed, the proof of the following lemma gives a sequence $\paren{f_{n}} \subseteq H_{+}^{AC}$ which converges to the identity uniformly but fails to converge in $H_{+}^{AC}$.

\begin{lem}\label{RhoDistanceBound}
    For any $f,g \in H_{+}^{AC} \paren{I}$ and $a,b \in I$, we have $\rho_{a}^{b} \paren{f,g} \leq f \paren{b} - f \paren{a} + g \paren{b} - g \paren{a}$. It follows that $\diam_{\rho} \paren{H_{+}^{AC} \paren{\brack{a,b}}} = 2 \paren{b-a}$.
\end{lem}
\begin{proof}
    For any $f \in H_{+}^{AC}$, we have $f' \geq 0$ at every point where $f'$ exists, since $f$ is increasing. Thus, for each $f,g \in H_{+}^{AC}$:
    \begin{align*}
        \rho_{a}^{b} \paren{f,g} &= \int_{a}^{b} \absval{f' - g'} \d \lambda\\
        &\leq \int_{a}^{b} \absval{f'} \d \lambda + \int_{a}^{b} \absval{g'} \d \lambda\\
        &= \int_{a}^{b} f' \; \d \lambda + \int_{a}^{b} g' \; \d \lambda\\
        &= f \paren{b} - f \paren{a} + g \paren{b} - g \paren{a}
    \end{align*}

    In the special case where $a,b$ are the endpoints of $I$, we have $\rho_{a}^{b} \paren{f,g} \leq 2 \paren{b-a}$ since $f,g$ fix the endpoints. Hence, $\diam_{\rho} \paren{H_{+}^{AC}} \leq 2 \paren{b-a}$.
    
    For the reverse inequality, we fix $I = \brack{0,1}$; the argument is easily adjusted for general $\brack{a,b}$. For each $n \geq 2$, let $f_{n}$ be the piecewise-linear map generated by the set of points: $$\brace{\paren{\frac{i}{n}, \frac{i}{n}} : 0 \leq i \leq n} \cup \brace{\paren{\frac{i}{n} + \frac{1}{n^{2}}, \frac{i+1}{n} - \frac{1}{n^{2}}} : 0 \leq i < n}.$$
    Piecewise-linear maps defined in this way are clearly absolutely continuous, so $\paren{f_{n}}_{n \geq 2} \subseteq H_{+}^{AC}$. We leave it as an exercise to the reader to calculate that $\rho \paren{f_{n}, \id_{I}} = 2 - \frac{4}{n}$, and so $\diam_{\rho} \paren{H_{+}^{AC}} \geq \lim_{n \to \infty} \rho \paren{f_{n},\id_{I}} = 2$.
\end{proof}

Observe that for $f \in H_{+}$, the set $I \setminus \Fix{f}$ is open, and its connected components may be categorized according to whether $f \paren{x} > x$ or $f \paren{x} < x$ for each point $x$ of the connected component---that is, whether the ``bumps'' lie above or below the diagonal. We introduce some useful notation to describe this behavior.

\begin{defn}
    Given $f \in H_{+} \paren{I}$ and $x \in I$, we define the \textbf{orbital} of $x$ under $f$, denoted $\orb{f}{x}$, as follows. If $x \in \Fix{f}$, we set $\orb{f}{x} := \brace{x}$; otherwise, we let $\orb{f}{x}$ be the connected component of $I \setminus \Fix{f}$ containing $x$.

    If $x \in \Fix{f}$, we say $\orb{f}{x}$ has parity zero; otherwise, we say $\orb{f}{x}$ has positive (resp. negative) parity if $f \paren{x} > x$ (resp. $f \paren{x} < x$). Similarly, we say a point has positive, negative, or zero parity if its orbital has the corresponding parity.
\end{defn}

Observe that the orbital $\orb{f}{x}$ is the convex hull of the orbit $f^{\Z} \paren{x} := \brace{f^{n} \paren{x} : n \in \Z}$; that is, we have $\orb{f}{x} = \paren{\inf f^{\Z} \paren{x}, \sup f^{\Z} \paren{x}}$ whenever $x \notin \Fix{f}$. We mention a consequence of this observation, which will be needed later:

\begin{lem}\label{NoCommonFixedPointsCanMovePts}
    Let $\Gamma \leq H_{+}$ be a subgroup with generating set $A \subseteq \Gamma$. Then for every $a \in I$, we have $\sup \brace{f \paren{a} : f \in \Gamma} \in \bigcap_{f \in A} \Fix{h}$.
\end{lem}
\begin{proof}
    Fix $a \in I$. Let $x := \sup \brace{f \paren{a} : f \in \Gamma}$, and suppose for contradiction that $x \notin \bigcap_{f \in A} \Fix{f}$. Take $f \in A$ so that $x \notin \Fix{f}$. Then the orbital $\orb{f}{x}$ is an open neighborhood of $x$. By choice of $x$, there is $h \in \Gamma$ such that $h \paren{a} \in \orb{f}{x}$. But then $f^{\Z} \paren{h \paren{a}}$ is cofinal in $\orb{f}{x}$, so there is $N \in \Z$ such that $f^{N} \paren{h \paren{a}} > x = \sup \brace{f \paren{a} : f \in \Gamma}$ --- a contradiction, since $f^{N}h \in \Gamma$.
\end{proof}

\section{Finitely generated subgroups}\label{sec:FinGenSubgps}
In the sequel, we fix $I := \brack{0,1}$ unless otherwise noted.

\subsection{Topological rank}
\begin{defn}
    For a Polish group $G$ and an $n \leq \aleph_{0}$, we define: $$\Omega_{n} := \brace{\paren{g_{i}}_{i < n} \in G^{n} : \textnormal{$\angle{g_{i} : i < n}$ is dense in $G$}}.$$

    We say $G$ is \textit{topologically $n$-generated} if there is a dense $n$-generated subgroup, i.e., if $\Omega_{n} \neq \emptyset$. Similarly, we say $G$ is \textit{generically $n$-generated} if the set $\Omega_{n}$ is comeagre in the product topology on $G^{n}$.

    We call the least $n$ for which $G$ is topologically (resp. generically) $n$-generated the \textit{topological rank} (resp. \textit{generic rank}), denoted $\trk{G}$ (resp. $\grk{G}$).
\end{defn}

We list some basic properties of topological and generic rank.

\begin{itemize}
    \item Since Polish groups are Baire, being generically $n$-generated implies being topologically $n$-generated. In particular, $\trk{G} \leq \grk{G}$.
    \item If $G$ is topologically (resp. generically) $n$-generated and $m \geq n$, then $G$ is topologically (resp. generically) $m$-generated. The case for ``topologically'' is immediate, and the case for ``generically'' follows from the Kuratowski-Ulam theorem.
    \item By second-countability, the set $\Omega_{n}$ is $G_{\delta}$.\footnote{This is because $\angle{g_{i} : i < n}$ is dense iff for every basic open set $U$, there exists a word $w$  in the free group on $n$ letters such that $w \paren{\paren{g_{i}}_{i < n}} \in U$. Each $w$ corresponds to a continuous map $G^{n} \to G$, by continuity of the group operations.} Then to show $G$ is generically $n$-generated, it suffices to show $\Omega_{n}$ is dense in $G^{n}$.
    \item Every Polish group is generically $\aleph_{0}$-generated. 
\end{itemize}

\begin{lem}\label{TopRankBounds}
    If $\phi \from G_{1} \to G_{2}$ is a continuous group homomorphism with dense image, we have $\trk{G_{2}} \leq \trk{G_{1}}$.
\end{lem}
\begin{proof}
    If $G_{1} = \clGrp{A}$ for some $A \subseteq G_{1}$ with $\absval{A} = \trk{G_{1}}$, we have $\phi \brack{\angle{A}} = \angle{\phi \brack{A}}$ since $\phi$ is a homomorphism, and $\phi \brack{\clGrp{A}} \subseteq \overline{\phi \brack{\angle{A}}}$ by continuity. Then $\clGrp{\phi \brack{A}} = \overline{\phi \brack{\angle{A}}} \supseteq \phi \brack{\clGrp{A}} = \phi \brack{G_{1}}$. Since the closure of $\angle{\phi \brack{A}}$ contains the dense set $\phi \brack{G_{1}}$, it follows that $\angle{\phi \brack{A}}$ is dense, i.e.\ $\phi \brack{A}$ generates a dense subgroup of $G_{2}$. Thus $\trk{G_{2}} \leq \absval{\phi \brack{A}} \leq \absval{A} = \trk{G_{1}}$.
\end{proof}

The following result is the current state of affairs for the topological rank of $H_{+}^{AC}$, given the existing literature.

\begin{cor}
    $2 \leq \trk{H_{+}^{AC}} \leq 10$.
\end{cor}
\begin{proof}
    Let $D_{+}^{1}$ denote the group of increasing $C^{1}$-diffeomorphisms of $I$, with the usual $C^{1}$ topology. Then the inclusion maps $D_{+}^{1} \subseteq H_{+}^{AC} \subseteq H_{+}$ are continuous and have dense image, so by \Cref{TopRankBounds}, we have $\trk{H_{+}} \leq \trk{H_{+}^{AC}} \leq \trk{D_{+}^{1}}$. We have $\trk{H_{+}} = 2$: each cyclic subgroup of $H_{+}$ is discrete, so $\trk{H_{+}} > 1$, but Thompson's group $F$ is $2$-generated and embeds in $H_{+}$ as a dense subgroup, so $\trk{H_{+}} = 2$. Moreover, Akhmedov \& Cohen showed in \cite{AC19} that $\trk{D_{+}^{1}} \leq 10$.
\end{proof}

\subsection{Genericity of generators}

The goal of the rest of this section is to show that $\grk{H_{+}^{AC}} = 2$. This is the best result possible, since we already have $2 \leq \trk{H_{+}^{AC}} \leq \grk{H_{+}^{AC}}$. 

As mentioned in the introduction, the main technique is similar to the proof that $\grk{H_{+}} = 2$, from \cite{AC19}. The main difference in the $H_{+}^{AC}$ case is the extra care that must be taken to ensure various approximations by ``nice'' maps can be made.

The following lemma shows that we may approximate an element of $H_{+}^{AC}$ by blowing up fixed points to entire fixed intervals.

\begin{lem}\label{CanStuffAFixedIntervalInThere}
    Let $\phi \in H_{+}^{AC}$, let $\paren{x_{i}}_{i < N}$ a collection of fixed points of $\phi$, and let $\paren{\brack{a_{i}, b_{i}}}_{i < N}$ be a pairwise-disjoint collection of closed intervals with $x_{i} \in \brack{a_{i},b_{i}}$ for each $i < N$. Then there is a map $\psi \in H_{+}^{AC}$ which fixes each $\brack{a_{i},b_{i}}$ pointwise, and satisfying: $$\rho_{I} \paren{\phi, \psi} \leq \sum_{i < N}\absval{a_{i} - \phi \paren{a_{i}}} + \absval{\phi \paren{b_{i}} - b_{i}} + \paren{b_{i} - a_{i}} + \paren{\phi \paren{b_{i}} - \phi \paren{a_{i}}}.$$
\end{lem}

Before proving this lemma, we remark that the right-hand side of the inequality is jointly continuous in the $a_{i}$'s and $b_{i}$'s, and converges to $0$ as the $a_{i}$'s and $b_{i}$'s approach the $x_{i}$'s. Thus, for every $\epsilon > 0$, there is $\gamma > 0$ small enough so that for \emph{every} choice of $\paren{\brack{a_{i},b_{i}}}_{i<N}$ satisfying $\max_{i < N} \paren{b_{i} - a_{i}} < \gamma$, the corresponding $\psi$ satisfies $\rho_{I} \paren{\phi,\psi} < \epsilon$.

\begin{proof}
    Define $\psi \in H_{+}^{AC}$ by: $$\psi \paren{x} := \begin{cases}
        \frac{a_{0}}{\phi \paren{a_{0}}} \phi \paren{x} & \text{if $x \in \left[ 0, a_{0} \right)$,}\\
        b_{i} + \frac{a_{i+1} - b_{i}}{\phi \paren{a_{i+1}} - \phi \paren{b_{i}}} \paren{\phi \paren{x} - \phi \paren{b_{i}}} & \text{if $x \in \paren{b_{i}, a_{i+1}}$ for some $i < N$,}\\
        b_{N-1} + \frac{1 - b_{N-1}}{1 - \phi \paren{b_{N-1}}} \paren{\phi \paren{x} - \phi \paren{b_{N-1}}} & \text{if $x \in \left( b_{N-1}, 1 \right]$,}\\
        x & \text{if $x \in \bigcup_{i<N} \brack{a_{i},b_{i}}$.}
    \end{cases}$$
    Absolute continuity is preserved by piecewise-defined constructions with finite pieces, as well as by linear combinations, so $\psi \in H_{+}^{AC}$.

    We have $\rho_{I} \paren{\phi, \psi} = \rho_{0}^{a_{0}} \paren{\phi, \psi} + \sum_{i < N} \rho_{a_{i}}^{b_{i}} \paren{\phi, \psi} + \sum_{i<N-1} \rho_{b_{i}}^{a_{i+1}} \paren{\phi, \psi} + \rho_{b_{N-1}}^{1} \paren{\phi, \psi}$. Then for each $i < N-1$:
    \begin{align*}
        \rho_{b_{i}}^{a_{i+1}} \paren{\phi, \psi} &= \int_{b_{i}}^{a_{i+1}} \absval{{\psi}' - \phi'} \d \lambda\\
        &= \int_{b_{i}}^{a_{i+1}} \absval{\frac{a_{i+1} - b_{i}}{\phi \paren{a_{i+1}} - \phi \paren{b_{i}}} \phi' - \phi'} \d \lambda\\
        &= \absval{\frac{a_{i+1} - b_{i}}{\phi \paren{a_{i+1}} - \phi \paren{b_{i}}} - 1} \int_{b_{i}}^{a_{i+1}} \absval{\phi'} \d \lambda\\
        &= \absval{\frac{a_{i+1} - b_{i}}{\phi \paren{a_{i+1}} - \phi \paren{b_{i}}} - 1} \paren{\phi \paren{a_{i+1}} - \phi \paren{b_{i}}}\\
        &= \absval{a_{i+1} - \phi \paren{a_{i+1}} + \phi \paren{b_{i}} - b_{i}}\\
        &\leq \absval{a_{i+1} - \phi \paren{a_{i+1}}} + \absval{\phi \paren{b_{i}} - b_{i}}
    \end{align*}
    Similar calculations show that $\rho_{0}^{a_{0}} \paren{\phi, \psi} = \absval{\phi \paren{a_{0}} - a_{0}}$ and $\rho_{b_{N-1}}^{1} \paren{\phi, \psi} = \absval{\phi \paren{b_{N-1}} - b_{N-1}}$. Furthermore, by \Cref{RhoDistanceBound}, we have $\rho_{a_{i}}^{b_{i}} \paren{\phi, \psi} \leq \paren{\psi \paren{b_{i}} - \psi \paren{a_{i}}} + \paren{\phi \paren{b_{i}} - \phi \paren{a_{i}}} = \paren{b_{i} - a_{i}} + \paren{\phi \paren{b_{i}} - \phi \paren{a_{i}}}$ for each $i$.
    The result follows.
\end{proof}

\begin{thm}\label{HACIsGen2Gen}
    $H_{+}^{AC}$ is generically $2$-generated.
\end{thm}
\begin{proof}
    Let $f,g \in H_{+}^{AC}$, and let $\delta > 0$. We show there are $\tilde{f}$ and $\tilde{g}$ such that $\rho \paren{f,\tilde{f}} < \delta$, $\rho \paren{g,\tilde{g}} < \delta$, and $\angle{\tilde{f}, \tilde{g}}$ is dense. It suffices to assume $f$ and $g$ have no fixed points in common other than $0$ and $1$, since the set of pairs $\paren{f,g}$ with this property is dense in $\paren{H_{+}^{AC}}^{2}$.

    By continuity, fix $\alpha > 0$ small enough so that $x \leq \alpha$ implies $f \paren{x} < \frac{\delta}{2}$ and $g \paren{x} < \frac{\delta}{2}$. We define $y_{0} \in I$ and $\tilde{g} \in H_{AC}^{+}$ as follows; let $y_{0} := \frac{1}{2} \min \brace{\alpha, g \paren{\alpha}, f \paren{\alpha}}$. Then define $\tilde{g}$ by setting $\tilde{g} \dhr_{\brack{\alpha,1}} = g \dhr_{\brack{\alpha,1}}$, and letting $\tilde{g} \dhr_{\brack{0,\alpha}}$ be the piecewise-linear map connecting the points $\paren{0,0}$, $\paren{\frac{y_{0}}{3}, \frac{2y_{0}}{3}}$, $\paren{y_{0}, y_{0}}$, and $\paren{\alpha, g \paren{\alpha}}$. Then $\tilde{g}$ is absolutely continuous, since $g$ is, and piecewise-linear maps are also. Furthermore, $\tilde{g} \dhr_{\brack{\alpha,1}} = g \dhr_{\brack{\alpha,1}}$ implies $\rho_{I} \paren{g,\tilde{g}} = \rho_{0}^{\alpha} \paren{g,\tilde{g}} \leq g \paren{\alpha} - g \paren{0} + \tilde{g} \paren{\alpha} - \tilde{g} \paren{0} = 2 g \paren{\alpha} < \delta$ by \Cref{RhoDistanceBound}.
    
    We now define $\tilde{f}$, as follows. Let $x_{0}$ be any element of $\paren{0,y_{0}}$, and for each $n > 0$, let $x_{n} := \tilde{g}^{-n} \paren{x_{0}}$. Since $\tilde{g} \paren{x} > x$ whenever $0 < x < y_{0}$, it follows that the sequence $\paren{x_{n}}_{n < \omega}$ strictly decreases and converges to $0$. Since $H_{+}^{AC} \paren{\brack{x_{1}, x_{0}}} \isom H_{+}^{AC} \paren{I}$, which is topologically $N$-generated for some $N$ (say, $N = 10$), we may take $\phi_{0}, \phi_{1}, \ldots, \phi_{N-1} \in H_{+}^{AC} \paren{\brack{x_{1}, x_{0}}}$ which generate a dense subgroup. We define $\tilde{f}$ by:
    \begin{itemize}
        \item $\tilde{f} \dhr_{\brack{\alpha,1}} = f \dhr_{\brack{\alpha,1}}$;
        \item $\tilde{f} \dhr_{\brack{x_{0}, \alpha}}$ is any absolutely continuous interpolating map between $\paren{x_{0}, x_{0}}$ and $\paren{\alpha, f \paren{\alpha}}$ which shares no fixed points with $\tilde{g}$;
        \item $\tilde{f} \dhr_{\brack{x_{kN + i + 1}, x_{kN + i}}} = \tilde{g}^{-(kN+i)} \circ \phi_{i} \circ \tilde{g}^{kN+i}$ for all $k < \omega$ and $0 \leq i < N$.
    \end{itemize}
    By the same argument used for $\tilde{g}$, since $\tilde{f}$ agrees with $f$ on $\brack{\alpha,1}$, we have $\rho_{I} \paren{\tilde{f},f} \leq 2 f \paren{\alpha} < \delta$.
    
    We claim that for any $n$, the set $\brace{\paren{\tilde{g}^{i} \circ \tilde{f} \circ \tilde{g}^{-i}} \dhr_{\brack{x_{n+1}, x_{n}}} : i < N}$ generates a dense subgroup of $H_{+}^{AC} \paren{\brack{x_{n+1}, x_{n}}}$. Indeed, one can show by induction that this set is equal to $\brace{\tilde{g}^{n} \circ \phi_{i} \circ \tilde{g}^{-n} : i < N}$. This is just the image of the set $\brace{\phi_{i} : i < N}$ under conjugation by $\tilde{g}^{n}$, which is a topological group isomorphism $H_{+}^{AC} \paren{\brack{x_{1},x_{0}}} \to H_{+}^{AC} \paren{\brack{x_{n+1}, x_{n}}}$.

    It remains to check that $\Gamma := \angle{\tilde{f}, \tilde{g}}$ is dense in $H_{+}^{AC}$. Fix $\phi \in H_{+}^{AC}$ and $\epsilon > 0$; we will find a map $W \in \Gamma$ with $\rho_{I} \paren{W,\phi} < \epsilon$. By applying \Cref{CanStuffAFixedIntervalInThere} (and the remark preceding its proof) to the fixed points $0$ and $1$, we may choose $\gamma > 0$ satisfying $\gamma < \min \brace{\frac{1}{2},\frac{\epsilon}{12}}$, and for all $a,b \in I$ with $a < \gamma$ and $1-\gamma < b$, we have $\rho_{I} \paren{\phi,\psi_{a,b}} < \frac{\epsilon}{3}$ where $\psi_{a,b} \in H_{+}^{AC}$ is the result of applying \Cref{CanStuffAFixedIntervalInThere} to the intervals $\brack{0,a}$ and $\brack{b,1}$.

    By construction, $\tilde{f}$ and $\tilde{g}$ do not share a fixed point in $\paren{0,1}$. Therefore by \Cref{NoCommonFixedPointsCanMovePts}, $\sup \brace{h \paren{y_{0}} : h \in \Gamma} = 1$, and so there is $h \in \Gamma$ for which $h \paren{y_{0}} > 1 - \gamma$. Set $F := h \tilde{f} h^{-1} \in \Gamma$ and $G := h \tilde{g} h^{-1} \in \Gamma$. Then we have $\Fix{F} = h \brack{\Fix{\tilde{f}}}$ and $\Fix{G} = h \brack{\Fix{\tilde{g}}}$.

    Choose $n$ large enough so that $h \paren{x_{n+1}} < \gamma$. Since $F$ has no fixed points on the interval $h \brack{\left( x_{0}, y_{0} \right]} = \left( h \paren{x_{0}}, h \paren{y_{0}} \right]$, and this interval contains $G \paren{h \paren{x_{0}}}$, we may find an integer $m$ such that $F^{m} \paren{G \paren{h \paren{x_{0}}}} > h \paren{y_{0}}$. Set $\Phi := G^{-(n+1)} F^{m} G^{n+1} (= h \tilde{g}^{-(n+1)} \tilde{f}^{m} \tilde{g}^{n+1} h^{-1}) \in \Gamma$. Then one calculates that $\Phi$ fixes $h \paren{x_{n+1}}$, and $\Phi \paren{h \paren{x_{n}}} > h \paren{y_{0}}$. Then if $a := \Phi \paren{h \paren{x_{n+1}}}$ and $b := \Phi \paren{h \paren{x_{n}}}$, we have: $$a = h \paren{x_{n+1}} < \gamma < 1-\gamma < h \paren{y_{0}} < b.$$

    Since $\brace{\paren{\tilde{g}^{i} \tilde{f} \tilde{g}^{-i}} \dhr_{\brack{x_{n+1}, x_{n}}} : i < N}$ generates a dense subgroup of $H_{+}^{AC} \paren{\brack{x_{n+1}, x_{n}}}$, then $\brace{\paren{\Phi G^{i} F G^{-i} \Phi^{-1}} \dhr_{\brack{a,b}} : i < N}$ generates a dense subgroup of $H_{+}^{AC} \paren{\brack{a,b}}$ (this is just applying the isomorphism of conjugation by $\Phi h$).

    Note that $\psi_{a,b} \dhr_{\brack{a,b}} \in H_{+}^{AC} \paren{\brack{a,b}}$ since $\psi_{a,b}$ fixes $a$ and $b$. Thus, let $w \in \mathbb{F}_{N}$ be a word on $N$ letters such that $\rho_{a}^{b} \paren{W, \psi_{a,b}} < \frac{\epsilon}{3}$, where $W := w \paren{\Phi F \Phi^{-1}, \ldots, \Phi G^{N-1} F G^{-(N-1)} \Phi^{-1}} \in \Gamma$. Then since $W$ and $\psi_{a,b}$ both fix $a$ and $b$, by \Cref{RhoDistanceBound} we have $\rho_{0}^{a} \paren{W,\psi_{a,b}} \leq W \paren{a} - W \paren{0} + \psi_{a,b} \paren{a} - \psi_{a,b} \paren{0} = 2a$, and similarly $\rho_{b}^{1} \paren{W,\psi_{a,b}} \leq 2 \paren{1-b}$. Thus:

    \begin{align*}
        \rho_{I} \paren{W,\phi} &\leq \rho_{I} \paren{W,\psi_{a,b}} + \rho_{I} \paren{\psi_{a,b}, \phi}\\
        &< \rho_{I} \paren{W,\psi_{a,b}} + \frac{\epsilon}{3}\\
        &= \rho_{0}^{a} \paren{W,\psi_{a,b}} + \rho_{a}^{b} \paren{W,\psi_{a,b}} + \rho_{b}^{1} \paren{W,\psi_{a,b}} + \frac{\epsilon}{3}\\
        &< \rho_{0}^{a} \paren{W,\psi_{a,b}} + \frac{\epsilon}{3} + \rho_{b}^{1} \paren{W,\psi_{a,b}} + \frac{\epsilon}{3}\\
        &\leq 2 \paren{a + (1-b)} + \frac{2\epsilon}{3}\\
        &< 2 \paren{2 \gamma} + \frac{2\epsilon}{3}\\
        &< 2 \paren{2 \frac{\epsilon}{12}} + \frac{2\epsilon}{3} = \epsilon. \qedhere
    \end{align*}

\end{proof}

\section{Conjugacy classes}\label{sec:ConjClasses}

The main result of this section is the existence of a comeagre conjugacy class of $H_{+}^{AC}$, as well as a characterization of its elements. For comparison, we first recall the corresponding result for $H_{+}$.

\begin{defn}
    We define the following subsets of $H_{+}$:
    \begin{itemize}
        \item $C$ is the set of all $f \in H_{+}$ for which the orbitals of nonzero parity form a dense linear order without endpoints (under the natural ordering), and moreover, the set of orbitals of positive parity and the set of orbitals of negative parity are dense subsets of this order.
        \item $P := \brace{f \in H_{+} : \textnormal{$\Fix{f}$ is a perfect set}}$.
        \item $T := \brace{f \in H_{+} : \textnormal{$\Fix{f}$ is totally disconnected}}$.
        \item $N := \brace{f \in H_{+} : \textnormal{$\Fix{f}$ is null}}$.
    \end{itemize}
\end{defn}

\begin{thm}[folklore; see \cite{Hjo00}]
    $C$ is a dense $G_{\delta}$ conjugacy class in $H_{+}$.
\end{thm}

We first show that $C \cap H_{+}^{AC}$ is comeagre in $H_{+}^{AC}$.

\begin{lem}\label{PerfectFixedSetIsGeneric}
    $P \cap H_{+}^{AC}$ is comeagre in $H_{+}^{AC}$.
\end{lem}
\begin{proof}
    Fix any countable open basis $\mathcal{U}$ for $I$. Given $U \in \mathcal{U}$, let $D_{U}$ denote the set of all $f \in H_{+}^{AC}$ such that $\Fix{f} \cap U$ is not a singleton. Then $P \cap H_{+}^{AC} = \bigcap_{U \in \mathcal{U}} D_{U}$, so it suffices to show each $D_{U}$ has dense interior.

    Fix $f \in H_{+}^{AC}$ and $\epsilon > 0$; we will show that $B_{\rho} \paren{f,\epsilon}$ intersects the interior of $D_{U}$. If $\Fix{g} \cap U = \emptyset$ for all $g$ with $\rho \paren{f,g} < \frac{\epsilon}{2}$, then $B_{\rho} \paren{f,\frac{\epsilon}{2}} \subseteq D_{U}$ and we are done. Otherwise, suppose there is $g \in B_{\rho} \paren{f,\frac{\epsilon}{2}}$ and $x_{0} \in I$ such that $x_{0} \in \Fix{g} \cap U$. By \Cref{CanStuffAFixedIntervalInThere}, we can take $a < x_{0} < b$ and $\psi \in H_{+}^{AC}$ so that $\brack{a,b} \subseteq U$ and $\rho \paren{g, \psi} + 2 \paren{b-a} < \frac{\epsilon}{2}$. We define $h \in H_{+}^{AC}$ by letting $h$ and $\psi$ agree on the set $I \setminus \paren{a,b}$, and letting $h \dhr_{\brack{a,b}}$ be any map in $H_{+}^{AC} \paren{\brack{a,b}}$ that satisfies $h \paren{y_{0}} > y_{0}$, $h \paren{y_{1}} < y_{1}$, and $h \paren{y_{2}} > y_{2}$ for some points $a < y_{0} < y_{1} < y_{2} < b$.\footnote{For example, we could let $h \dhr_{\brack{a,b}}$ be the piecewise linear map connecting the points $\paren{a,a}$, $\paren{\frac{3a+b}{4}, \frac{2a+b}{3}}$, $\paren{\frac{a+b}{2}, \frac{3a+2b}{5}}$, $\paren{\frac{a+3b}{4}, \frac{a+4b}{5}}$, and $\paren{b,b}$.}

    Note that $\psi$ and $h$ differ only on the set $\paren{a,b}$, so $\rho_{I} \paren{\psi, h} = \rho_{a}^{b} \paren{\psi, h} \leq 2 \paren{b-a}$. Thus, $\rho \paren{f,h} \leq \rho \paren{f,g} + \rho \paren{g,\psi} + \rho \paren{\psi, h} \leq \frac{\epsilon}{2} + \rho \paren{g,\psi} + 2 \paren{b-a} < \frac{\epsilon}{2} + \frac{\epsilon}{2} = \epsilon$, and so $h \in B_{\rho} \paren{f,\epsilon}$. Moreover, we claim $h \in \intr \paren{D_{U}}$; indeed, since the topology on $H_{+}^{AC}$ is finer than the uniform convergence topology, take $\delta > 0$ small enough so that $\rho \paren{h,k} < \delta$ implies $k \paren{y_{0}} > y_{0}$, $k \paren{y_{1}} < y_{1}$, and $k \paren{y_{2}} > y_{2}$. Then for all $k \in B_{\rho} \paren{h,\delta}$, by the Intermediate Value Theorem there will be points $x_{1}, x_{2} \in \Fix{k}$ such that $y_{0} < x_{1} < y_{1} < x_{2} < y_{2}$. Then $x_{1},x_{2} \in \paren{y_{0},y_{2}} \subseteq \brack{a,b} \subseteq U$ witnesses $k \in D_{U}$. It follows that $B_{d} \paren{h,\delta} \subseteq D_{U}$, and so $h \in B_{d} \paren{f,\epsilon} \cap \intr \paren{D_{U}}$ as desired.
\end{proof}

\begin{lem}\label{TDFixedSetIsGeneric}
    $T \cap H_{+}^{AC}$ is dense $G_{\delta}$ in $H_{+}^{AC}$.
\end{lem}
\begin{proof}
    We claim $T$ is $G_{\delta}$. A closed subset of $\R$ is totally disconnected iff it is nowhere dense. Thus, if $\mathcal{U}$ is a countable open basis for $I$, we have $f \in T$ if and only if $U \nsubseteq \Fix{f}$ for all $U \in \mathcal{U}$.  Thus it suffices to show $T_{U} := \brace{f \in H_{+} : U \nsubseteq \Fix{f}}$ is open in $H_{+}$. Fix $U \in \mathcal{U}$, and let $f \in T_{U}$. Then there is some $x_{0} \in U \setminus \Fix{f}$, i.e. $f \paren{x_{0}} \neq x_{0}$. We may choose $\delta > 0$ so that $\norm{f - g}_{\infty} < \delta$ implies $g \paren{x_{0}} \neq x_{0}$. Then $g \in T_{U}$; it follows that $B_{\mathrm{unif}} \paren{f,\delta} \subseteq T_{U}$, so $T_{U}$ is open as desired.

    Then $T$ is $G_{\delta}$ in $H_{+}$, and so $T \cap H_{+}^{AC}$ is $G_{\delta}$ in $H_{+}^{AC}$. Then density is immediate from observing that $T$ contains the set of all non-identity, strictly-increasing polynomials passing through $\paren{0,0}$ and $\paren{1,1}$, which is dense in $H_{+}^{AC}$.
\end{proof}

\begin{cor}
    For comeagrely-many $f \in H_{+}^{AC}$, $\brace{\orb{f}{x} : \textnormal{$x$ has non-zero parity}}$ is a dense linear order without endpoints.
\end{cor}
\begin{proof}
    It is a straightforward exercise to check that this is equivalent to saying $\Fix{f}$ is a Cantor set. A compact subset of $\R$ is a Cantor set iff it is perfect, non-empty, and totally disconnected. $\Fix{f}$ is always non-empty (it contains $0$ and $1$), it is perfect for the generic $f$ by \Cref{PerfectFixedSetIsGeneric}, and it is totally disconnected for the generic $f$ by \Cref{TDFixedSetIsGeneric}.
\end{proof}

\begin{thm}\label{GammaIsGenericInAC}
    $C \cap H_{+}^{AC}$ is dense $G_{\delta}$ in $H_{+}^{AC}$.
\end{thm}
\begin{proof}
    Since $C$ is $G_{\delta}$ in $H_{+}$, we have that $C \cap H_{+}^{AC}$ is $G_{\delta}$ in $H_{+}^{AC}$, so we need only show $C \cap H_{+}^{AC}$ is dense.

    By the previous corollary, it now suffices to show that for a dense set of $f$, the orbitals of each non-zero parity are dense in the space of non-zero orbitals. This is equivalent to $f \in \bigcap_{q,r \in \Q \cap I} E_{q,r}$, where we define $E_{q,r}$ to be the set of $f \in H_{+}^{AC}$ such that if $q \notin \Fix{f}$, $r \notin \Fix{f}$, $\orb{f}{q} < \orb{f}{r}$, then there exist $s,t \in \Q$ such that $\orb{f}{q} < \orb{f}{s} < \orb{f}{r}$, $\orb{f}{q} < \orb{f}{t} < \orb{f}{r}$, $s$ has positive parity, and $t$ has negative parity. Indeed, by a similar argument to the one used in \Cref{PerfectFixedSetIsGeneric}, one may show that each $E_{q,r}$ has dense interior.
\end{proof}

\begin{thm}\label{NullFixedPtSetIsGeneric}
    $N \cap H_{+}^{AC}$ is dense $G_{\delta}$ in $H_{+}^{AC}$.
\end{thm}
\begin{proof}
    It suffices to show that $F_{\epsilon} := \brace{f \in H_{+} : \lambda \paren{\Fix{f}} < \epsilon}$ is open in $H_{+}$, since then $N = \bigcap_{\epsilon \in \Q_{>0}} F_{\epsilon}$ is $G_{\delta}$ in $H_{+}$, and so $N \cap H_{+}^{AC}$ is $G_{\delta}$ in $H_{+}^{AC}$.

    Let $f \in F_{\epsilon}$. By regularity of Lebesgue measure, there is an open set $U \subseteq I$ with $\lambda \paren{U} < \epsilon$ and $\Fix{f} \subseteq U$. Since $I \setminus U$ is closed and disjoint from $\Fix{f}$, there is some $\delta > 0$ such that $\absval{f \paren{x} - x} \geq \delta$ for all $x \in I \setminus U$. Let $g \in H_{+}$ with $\norm{f - g}_{\infty} \leq \frac{\delta}{2}$; then we have $x \in I \setminus U \implies \absval{g \paren{x} - x} \geq \absval{f \paren{x} - x} - \absval{f \paren{x} - g \paren{x}} \geq \delta - \frac{\delta}{2} = \frac{\delta}{2} > 0 \implies x \notin \Fix{g}$. Thus $\Fix{g} \subseteq U$, so that $\lambda \paren{\Fix{g}} \leq \lambda \paren{U} < \epsilon$, and so $g \in F_{\epsilon}$. Hence $F_{\epsilon}$ is open, as desired.

    $N$ contains the set of all non-identity, strictly-increasing polynomials. Indeed, if $p$ is such a polynomial, then it may have at most $\deg \paren{p}$ fixed points. Finite sets are null, so $p \in N$. Such polynomials are dense in $H_{+}^{AC}$, so $N \cap H_{+}^{AC}$ is dense also.
\end{proof}

The following result is an analogue for $H_{+}^{AC}$ of a standard fact about $H_{+}$. We omit the proof, as it is essentially the same as for the $H_{+}$.

\begin{lem}\label{BumpsAreConjugate}
    If $f \in H_{+}^{AC} \paren{\brack{a,b}}$ and $g \in H_{+}^{AC} \paren{\brack{c,d}}$ have no fixed points other than the endpoints, and their non-zero orbitals have the same parity, then there is an absolutely continuous bijection $h \from \brack{a,b} \to \brack{c,d}$ which conjugates $f$ to $g$ (that is, $h \circ f = g \circ h$).
\end{lem}

We are now ready to prove the main result of this section.

\begin{thm}\label{HACHasGenerics}
    $C^{AC} := C \cap N \cap H_{+}^{AC}$ is a dense $G_{\delta}$ conjugacy class of $H_{+}^{AC}$.
\end{thm}
\begin{proof}
    $C \cap H_{+}^{AC}$ is dense $G_{\delta}$ by \Cref{GammaIsGenericInAC}, and $N \cap H_{+}^{AC}$ is dense $G_{\delta}$ by \Cref{NullFixedPtSetIsGeneric}. Thus $C^{AC}$ is dense $G_{\delta}$. Moreover, let $f \in C^{AC}$ and $h \in H_{+}^{AC}$. Then $f \in C$ implies $hfh^{-1} \in C$ since $C$ is a conjugacy class in $H_{+}$, and $f \in N$ implies $hfh^{-1} \in N$ since $\Fix{hfh^{-1}} = h \brack{\Fix{f}}$, and absolutely continuous maps send null sets to null sets. Thus $C^{AC}$ is conjugacy invariant, and so it remains to show that any two elements of $C^{AC}$ are conjugate in $H_{+}^{AC}$.

    Fix $f,g \in C^{AC}$. Let $\paren{\paren{a_{n}, b_{n}}}_{n < \omega}$ be an enumeration of the connected components of $I \setminus \Fix{f}$. Since $f,g \in C$ are conjugate in $H_{+}$, there is a map $\hat{h} \in H_{+}$ with $\hat{h} \circ f = g \circ \hat{h}$. In particular, $\hat{h} \brack{\Fix{f}} = \Fix{g}$, and $\parity{\paren{a_{n},b_{n}}, f} = \parity{\paren{\hat{h} \paren{a_{n}}, \hat{h} \paren{b_{n}}}, g}$ for each $n < \omega$. Then by \Cref{BumpsAreConjugate}, for each $n$ there is an absolutely continuous homeomorphism $h_{n} \from \brack{a_{n}, b_{n}} \to \brack{\hat{h} \paren{a_{n}}, \hat{h} \paren{b_{n}}}$ which satisfies $h_{n} \circ f \dhr_{\brack{a_{n}, b_{n}}} = g \dhr_{\brack{\hat{h} \paren{a_{n}}, \hat{h} \paren{b_{n}}}} \circ h_{n}$. Then define $h \in H_{+}$ by $h \paren{x} := \begin{cases} h_{n} \paren{x} & \text{if $x \in \paren{a_{n}, b_{n}}$ for $n < \omega$,}\\ \hat{h} \paren{x} & \text{if $x \in \Fix{f}$.}\end{cases}$

    For $x \in I$, we have either $x \in \paren{a_{n}, b_{n}}$ for some $n$, so $h \paren{f \paren{x}} = h_{n} \paren{f \paren{x}} = g \paren{h_{n} \paren{x}} = g \paren{h \paren{x}}$; or otherwise $x \in \Fix{f}$, in which case $h \paren{f \paren{x}} = h \paren{x} = \hat{h} \paren{x} = g \paren{\hat{h} \paren{x}} = g \paren{h \paren{x}}$. In either case, $h \circ f = g \circ h$, i.e.\ $f$ and $g$ are conjugate.

    $h \in H_{+}$, so it is differentiable almost everywhere and its derivative is $L_{1}$. We claim $h$ is absolutely continuous, i.e., that $\int_{0}^{x} h' \paren{t} \d t = h \paren{x}$ for all $x \in I$. We need only check that this holds for $x \in \Fix{f}$; then if instead $x \in \paren{a_{n},b_{n}}$ for some $n$, we have $\int_{0}^{x} h' \paren{t} \d t = \int_{0}^{a_{n}} h' \paren{t} \d t + \int_{a_{n}}^{x} h' \paren{t} \d t = h \paren{a_{n}} + \paren{h_{n} \paren{x} - h_{n} \paren{a_{n}}} = h \paren{x}$ since $a_{n} \in \Fix{f}$.
    
    Thus, fix $x \in \Fix{f}$ and let $E := \brace{n < \omega : \paren{a_{n}, b_{n}} \subseteq \brack{0,x}}$. Then $\brack{0,x} \setminus \Fix{f} = \bigsqcup_{n \in E} \paren{a_{n}, b_{n}}$, and it follows that $h \brack{\bigsqcup_{n \in E} \paren{a_{n}, b_{n}}} = \brack{0,h \paren{x}} \setminus \Fix{g}$ since $g \circ h = h \circ f$. Then since $\Fix{f}$ and $\Fix{g}$ are both null, we have:
    \begin{align*}
        \int_{0}^{x} h' \paren{t} \d t &= \int_{\Fix{f} \cap \brack{0,x}} h' \paren{t} \d t + \sum_{n \in E} \int_{a_{n}}^{b_{n}} h' \paren{t} \d t\\
        &= \sum_{n \in E} \int_{a_{n}}^{b_{n}} h_{n}' \paren{t} \d t\\
        &= \sum_{n \in E} h_{n} \paren{a_{n}} - h_{n} \paren{b_{n}}\\
        &= \lambda \paren{\bigsqcup_{n \in E} \paren{h_{n} \paren{a_{n}}, h_{n} \paren{b_{n}}}}\\
        &= \lambda \paren{\bigsqcup_{n \in E} h \brack{\paren{a_{n}, b_{n}}}}\\
        &= \lambda \paren{h \brack{\bigsqcup_{n \in E} \paren{a_{n}, b_{n}}}}\\
        &= \lambda \paren{\brack{0,h \paren{x}} \setminus \Fix{g}}\\
        &= \lambda \paren{\brack{0,h \paren{x}}}\\
        &= h \paren{x}.
    \end{align*}
    The proof that $h^{-1}$ is absolutely continuous proceeds similarly, by interchanging the roles of $f$ and $g$. Thus $h \in H_{+}^{AC}$ conjugates $f$ to $g$, as desired.
\end{proof}

\printbibliography

\end{document}